\documentclass[10pt]{amsart}
\usepackage{amssymb}
\usepackage{bm}
\usepackage{graphicx}
\usepackage[centertags]{amsmath}
\usepackage{amsfonts}
\usepackage{amsthm}
\newtheorem{thm}{Theorem}
\newtheorem{cor}[thm]{Corollary}
\newtheorem{lem}[thm]{Lemma}
\newtheorem{prop}[thm]{Proposition}
\newtheorem{fact}[thm]{Fact}
\newtheorem{claim}[thm]{Claim}
\newtheorem{defn}[thm]{Definition}
\theoremstyle{definition}

\newtheorem{examp}{Example}
\newcommand{\rr}{\mathbb{R}}
\newcommand{\nn}{\mathbb{N}}

\newcommand{\ee}{\varepsilon}
\newcommand{\con}{\smallfrown}
\newcommand{\kk}{\mathcal{K}}
\newcommand{\ff}{\mathcal{F}}
\newcommand{\SB}{\mathbf{\Sigma}}
\newcommand{\PB}{\mathbf{\Pi}}

\newcommand{\bt}{\mathbb{N}^{<\mathbb{N}}}
\newcommand{\ct}{2^{<\mathbb{N}}}
\newcommand{\tr}{\mathrm{Tr}}
\newcommand{\wf}{\mathrm{WF}}
\newcommand{\sg}{\sigma}
\newcommand{\bb}{\mathcal{B}}
\newcommand{\SRC}{\mathrm{SRC}}
\newcommand{\bs}{\mathbf{f}=(f_n)}
\newcommand{\bsg}{\mathbf{g}=(g_n)}
\newcommand{\bff}{\mathbf{f}}
\newcommand{\bgg}{\mathbf{g}}

\newcommand{\ltr}{\Lambda^{<\nn}}
\newcommand{\FIN}{\mathrm{FIN}}

\begin{document}

\title[A strong boundedness result]{A strong boundedness result for
separable Rosenthal compacta}
\author{Pandelis Dodos}
\address{Universit\'{e} Pierre et Marie Curie - Paris 6, Equipe d' Analyse
Fonctionnelle, Bo\^{i}te 186, 4 place Jussieu, 75252 Paris Cedex 05, France.}
\email{pdodos@math.ntua.gr}


\maketitle


\begin{abstract}
It is proved that the class of separable Rosenthal compacta on the
Cantor set having a uniformly bounded dense sequence of continuous
functions, is strongly bounded.
\end{abstract}


\section{Introduction}

Our main result is a strong boundedness result for the class of
separable Rosenthal compacta (that is, separable compact subsets
of the first Baire class -- see \cite{ADK} and \cite{Ro2}) on the
Cantor set having a uniformly bounded dense sequence of continuous
functions. We shall denote this class by $\SRC$. The phenomenon of
strong boundedness, which was first touched by A. S. Kechris and
W. H. Woodin in \cite{KW}, is a strengthening of the classical
property of boundedness of $\PB^1_1$-ranks. Abstractly,
one has a $\PB^1_1$ set $B$, a natural notion of embedding between
elements of $B$ and a canonical $\PB^1_1$-rank $\phi$ on $B$ which is
coherent with the embedding, in the sense that if $x, y\in B$ and
$x$ embeds into $y$, then $\phi(x)\leq \phi(y)$. The strong
boundedness of $B$ is the fact that for every analytic subset $A$ of $B$
there exists $y\in B$ such that $x$ embeds into $y$ for every $x\in A$.
Basic examples of strongly bounded classes are the well-orderings
$\mathrm{WO}$ and the well-founded trees $\mathrm{WF}$ (although, in
these cases strong boundedness is easily seen to be equivalent to
boundedness). Recently, it was shown (see \cite{AD} and
\cite{DF}) that several classes of separable Banach spaces
are strongly bounded, where the corresponding notion of embedding
is that of (linear) isomorphic embedding. These results have, in turn,
important consequences in the study of universality problems in Banach
space Theory.

We will add another example to the list of strongly bounded classes,
namely the class $\SRC$. We notice that every $\kk$ in $\SRC$ can be
naturally coded by its dense sequence of continuous functions.
Hence, we identify $\SRC$ with the set
\[ \big\{ (f_n)\in B(2^\nn)^\nn: \overline{\{f_n\}}^p\subseteq
\bb_1(2^\nn) \text{ and } f_n\neq f_m \text{ if } n\neq m \big\} \]
where $B(2^\nn)$ stands for the closed unit ball of the separable
Banach space $C(2^\nn)$. With this identification, the set
$\SRC$ is $\PB^1_1$-true. A canonical $\PB^1_1$-rank on $\SRC$ comes
from the work of H. P. Rosenthal. Specifically, for every $\bs$ in
$\SRC$ one is looking at the order of the $\ell_1$-tree of the
sequence $(f_n)$. One has also a natural notion of topological
embedding between elements of $\SRC$. In particular, if $\bs$ and $\bsg$
are in $\SRC$, then we say that $\bgg$ \textit{topologically embeds} into $\bff$,
if there exists a homeomorphic embedding of the compact $\overline{\{g_n\}}^p$
into $\overline{\{f_n\}}^p$. This topological embedding, however,
is rather weak and not coherent with the $\PB^1_1$-rank on $\SRC$. Thus,
we strengthen the notion of embedding by imposing extra metric conditions
on the relation between $\bgg$ and $\bff$. To motivate our definition,
assume that $\bgg=(g_n)$ and $\bff=(f_n)$ were in addition Schauder basic
sequences. In this case the most natural thing to consider is equivalence
of basic sequences, i.e. $\bgg$ embeds into $\bff$ if there exists $L=
\{l_0<l_1<...\}\in[\nn]$ such that $(g_n)$ is equivalent to
$(f_{l_n})$. In such a case, it is easily seen the order of
the $\ell_1$-tree of $\bgg$ is dominated by the
one of $\bff$.

Although not every sequence $\bff\in\SRC$ is Schauder basic, the following condition
incorporates the above observation. So, we say that $\bgg=(g_n)$ \textit{strongly
embeds} into $\bff=(f_n)$, if $\bgg$ topologically embeds into $\bff$ and, moreover,
for every $\ee>0$ there exists $L_\ee=\{l_0<l_1<...\}\in[\nn]$ such that for every
$k\in\nn$ and every $a_0,...,a_k\in \rr$ we have
\[ \Big| \max_{0\leq i\leq k} \big\| \sum_{n=0}^i a_n g_n\big\|_{\infty} -
\big\| \sum_{n=0}^k a_n f_{l_n}\big\|_{\infty} \Big| \leq \ee \sum_{n=0}^k
\frac{|a_n|}{2^{n+1}}.\]
The notion of strong embedding is coherent with the $\PB^1_1$-rank on
$\SRC$ and is consistent with our motivating observation, in the sense
that if $\bgg=(g_n)$ strongly embeds into $\bff=(f_n)$ and $(g_n)$ is
Schauder basic, then there exists $L=\{l_0<l_1<...\}\in [\nn]$ such that
$(f_{l_n})$ is Schauder basic and equivalent to $(g_n)$. Under the above
terminology, we prove the following.
\medskip

\noindent \textbf{Main Theorem.} \textit{Let $A$ be an analytic subset
of $\SRC$. Then there exists $\bff\in\SRC$ such that for every $\bgg\in A$
the sequence $\bgg$ strongly embeds into $\bff$.}


\section{Background material}


We let $\nn=\{0,1,2,...\}$. By $[\nn]$ we denote the set of all infinite
subsets of $\nn$, while for every $L\in[\nn]$ by $[L]$ we denote the set
of all infinite subsets of $L$. For every Polish space $X$ by $\bb_1(X)$
we denote the set of all real-valued, Baire-1 functions on $X$.
If $\mathcal{F}$ is a subset of $\rr^X$, then by $\overline{\mathcal{F}}^p$
we shall denote the closure of $\mathcal{F}$ in $\rr^X$.

Our descriptive set theoretic notation and terminology
follows \cite{Kechris}. If $X, Y$ are Polish spaces, $A\subseteq X$ and
$B\subseteq Y$, then we say that $A$ is \textit{Wadge} (respectively
\textit{Borel}) \textit{reducible} to $B$ if there exists a continuous
(respectively Borel) map $f:X\to Y$ such that $f^{-1}(B)=A$. If $A$
is $\PB^1_1$, then a map $\phi:A\to\omega_1$ is said to be a
$\PB^1_1$-\textit{rank} on $A$ if there exist relations
$\leq_\Sigma, \leq_\Pi$ in $\SB^1_1$ and $\PB^1_1$ respectively
such that for all $y\in A$ we have
\[ x\in A \text{ and } \phi(x)\leq \phi(y) \Leftrightarrow
x\leq_\Sigma y \Leftrightarrow x\leq_\Pi y. \]
We notice that if $B$ is Borel reducible to a set $A$ via a Borel map $f$ and
$\phi$ is a $\PB^1_1$-rank on $A$, then the map $\psi:B\to\omega_1$ defined by
$\psi(y)=\phi\big(f(x)\big)$ for all $y\in B$ is a $\PB^1_1$-rank on $B$.

\subsection{Trees} Let $\Lambda$ be a non-empty set. By $\ltr$ we denote
the set of all finite sequences of $\Lambda$. We view $\ltr$ as a tree
equipped with the (strict) partial order $\sqsubset$ of end-extension. If
$t\in \ltr$, then the length $|t|$ of $t$ is defined to be the
cardinality of the set $\{s\in\ltr:s\sqsubset t\}$. If $s,t\in \ltr$, then
by $s^\con t$ we denote their concatenation. Two nodes $s,t\in\ltr$ are said to
be \textit{comparable} if either $s\sqsubseteq t$ or $t\sqsubseteq s$; otherwise
are said to be \textit{incomparable}. A subset of $\ltr$ consisting of
pairwise comparable nodes is said to be a \textit{chain}. If $\Lambda=\nn$
and $L\in[\nn]$, then by $\FIN(L)$ we denote the subset of $L^{<\nn}$
consisting of all finite \textit{strictly increasing} sequences in $L$.
For every $x\in \Lambda^\nn$ and every $n\geq 1$ we set
$x|n=\big( x(0),..., x(n-1)\big)\in \ltr$ while $x|0=\varnothing$.

A \textit{tree} $T$ on $\Lambda$ is a downwards closed subset of $\ltr$.
By $\tr(\Lambda)$ we denote the set of all trees on $\Lambda$. Hence
\[ T\in \tr(\Lambda) \Leftrightarrow \forall s,t\in\ltr \ (t\in T \
\wedge s\sqsubseteq t \Rightarrow s\in T). \]
A tree $T$ on $\Lambda$ is said to be \textit{pruned} if for every
$t\in T$ there exists $s\in T$ with $t\sqsubset s$. If $T\in\tr(\Lambda)$,
then the \textit{body} $[T]$ of $T$ is defined to be the set
$\{x\in\Lambda^\nn: x|n\in T \ \forall n\}$. A tree $T$ is said to
be \textit{well-founded} if $[T]=\varnothing$. The subset of
$\tr(\Lambda)$ consisting of all well-founded trees on $\Lambda$
will be denoted by $\wf(\Lambda)$. If $T\in\wf(\Lambda)$, we let
$T'=\{t:\exists s\in T \text{ with } t\sqsubset s\}\in\wf(\Lambda)$.
By transfinite recursion we define the iterated derivatives $T^{(\xi)}$
of $T$. The \textit{order} $o(T)$ of $T$ is defined to be the least
ordinal $\xi$ such that $T^{(\xi)}=\varnothing$. If $S, T$ are
well-founded trees, then a map $\phi:S\to T$ is called \textit{monotone}
if $s_1\sqsubset s_2$ in $S$ implies that $\phi(s_1)\sqsubset\phi(s_2)$
in $T$. Notice that in this case $o(S)\leq o(T)$. If $\Lambda, M$ are
non-empty sets, then we identify every tree $T$ on $\Lambda\times M$
with the set of all pairs $(s,t)\in\ltr\times M^{<\nn}$ such
that $|s|=|t|=k$ and $\big( (s(0),t(0)),....,(s(k-1),t(k-1))\big)\in T$.
If $\Lambda=\nn$, then we shall simply denote by $\tr$ and $\wf$ the
sets of all trees and well-founded trees on $\nn$ respectively. For
every countable set $\Lambda$ the set $\wf(\Lambda)$ is
$\PB^1_1$-complete and the map $T\to o(T)$ is a
$\PB^1_1$-rank on $\wf(\Lambda)$ (see \cite{Kechris}).

\subsection{Schauder basic sequences} A sequence $(x_n)$
of non-zero vectors in a Banach space $X$ is said to be a
\textit{Schauder basic sequence} if it is a Schauder basis
of its closed linear span (see \cite{LT}). This is equivalent
to say that there exists a constant $K\geq 1$ such that for
every $m,k\in\nn$ with $m< k$ and every $a_0,...,a_k\in\rr$ we have
\begin{equation}
\label{eSchauder} \big\| \sum_{n=0}^m a_n x_n \big\| \leq K \big\|
\sum_{n=0}^k a_n x_n \big\|.
\end{equation}
The least constant $K$ for which inequality (\ref{eSchauder})
holds is called the \textit{basis constant} of $(x_n)$.
A Schauder basic sequence $(x_n)$ is said to be \textit{monotone}
if $K=1$. It is said to be \textit{seminormalized} (respectively
\textit{normalized}) if there exists $M>0$ such that
$\frac{1}{M}\leq \|x_n\|\leq M$ (respectively $\|x_n\|=1$)
for every $n\in\nn$.

Let $X$ and $Y$ be Banach spaces. If $(x_n)$ and $(y_n)$ are
two sequences in $X$ and $Y$ respectively and $C\geq 1$, then
we say that $(x_n)$ is $C$-\textit{equivalent} to $(y_n)$ (or
simply equivalent, if $C$ is understood) if for every $k\in\nn$
and every $a_0,...,a_k\in\rr$ we have
\[ \frac{1}{C} \big\| \sum_{n=0}^k a_n y_n \big\|_Y
\leq \|\sum_{n=0}^k a_n x_n \big\|_X \leq
C \big\| \sum_{n=0}^k a_n y_n\big\|_Y. \]
We denote by $(x_n)\stackrel{C}{\sim}(y_n)$ the fact
that $(x_n)$ is $C$-equivalent to $(y_n)$.


\section{Coding $\SRC$}

Let $X$ be a compact metrizable space and let $\SRC(X)$ be the
family of all separable Rosenthal compacta on $X$ having a dense
set of  continuous functions which is uniformly bounded with
respect to  the supremum norm. We denote by $B(X)$ the closed unit
ball of the separable Banach space $C(X)$. Notice that every
$\kk\in \SRC(X)$ is naturally coded by its dense sequence of
continuous functions. Hence we may identify $\SRC(X)$ with the set
\[ \big\{ (f_n)\in B(X)^\nn: \overline{\{f_n\}}^p\subseteq
\bb_1(X) \text{ and } f_n\neq f_m \text{ if } n\neq m \big\}. \]
Let us denote by $\mathbf{B}(X)$ the $G_\delta$ subset of $B(X)^\nn$
consisting of all sequences $\bs$ in $B(X)^\nn$ such that
$f_n\neq f_m$ if $n\neq m$. With the above identification the set
$\SRC(X)$ becomes a subset of the Polish space $\mathbf{B}(X)$.
Moreover, as for every compact metrizable space $X$ the Banach
space $C(X)$ embeds isometrically into $C(2^\nn)$, we shall
denote by $\SRC$ the set $\SRC(2^\nn)$ and we view $\SRC$ as
the set of all separable Rosenthal compacta having a uniformly
bounded dense sequence of continuous functions and defined on
a compact metrizable space (it is crucial that $C(X)$ embeds
isometrically into $C(2^\nn)$ -- this will be clear later on).
The following lemma provides an estimate for the complexity
of the set $\SRC(X)$.
\begin{lem}
\label{srcl1} For every compact metrizable space $X$ the set
$\SRC(X)$ is $\PB^1_1$. Moreover, the set $\SRC$ is $\PB^1_1$-true.
\end{lem}
\begin{proof}
Instead of calculating the complexity of $\SRC(X)$ we will
actually find a Borel map $\Phi:\mathbf{B}(X)\to\tr$ such that
$\Phi^{-1}(\wf)=\SRC(X)$. In other words, we will find a Borel
reduction of $\SRC(X)$ to $\wf$. This will not only show that
$\SRC(X)$ is $\PB^1_1$, but also, it will provide a natural
$\PB^1_1$-rank on $\SRC(X)$. This canonical reduction comes
from the work of H. P. Rosenthal.

Specifically, let $(e_i)$ be the standard basis of $\ell_1$.
For every $d\in\nn$ with $d\geq 1$ and every $\bs$ in
$\mathbf{B}(X)$ we associate a tree $T^d_{\bff}$
on $\nn$ defined by
\[ s\in T^d_{\bff} \Leftrightarrow s=(n_0<...<n_k)\in \FIN(\nn)
\text{ and } (e_i)_{i=0}^{k} \stackrel{d}{\sim} (f_{n_i})_{i=0}^k. \]
Notice that $(e_i)_{i=0}^k \stackrel{d}{\sim} (f_{n_i})_{i=0}^k$
if for every $a_0,..., a_k\in\rr$ we have
\[ \frac{1}{d} \sum_{i=0}^k |a_i| \leq \big\| \sum_{i=0}^k a_i f_{n_i}
\big\|_{\infty} \leq d \sum_{i=0}^k |a_i|. \]
Observe that for every $t\in\bt$ the set $\{\mathbf{f}: t\in T^d_{\bff}\}$
is a closed subset of $\mathbf{B}(X)$. This yields that
the map $\mathbf{B}(X)\ni \mathbf{f}\mapsto T^d_{\bff}\in\tr$
is Borel (actually it is Baire-1). Next we glue the sequence of
trees $\{ T^d_{\bff}:d\geq 1\}$ and we obtain a tree $T_{\bff}$
on $\nn$ defined by the rule
\[ s\in T_\bff \Leftrightarrow \exists d\geq 1 \ \exists s'
\text{ with } s=d^\con s' \text{ and } s'\in T^d_\bff. \]
The tree $T_\bff$ is usually called the $\ell_1$-tree of the
sequence $\bs$. Clearly the map $\Phi:\mathbf{B}(X)\to\tr$
defined by $\Phi(\bff)=T_\bff$ is Borel.

We observe that
\[ \bs\in \SRC(X) \Leftrightarrow T_\bff\in \wf.\]
This equivalence is essentially Rosenthal's Dichotomy \cite{Ro1}
(see also \cite{Kechris} and \cite{To2}). Indeed, let $\bs$ be
such that $T_\bff$ is well-founded. By Rosenthal's Dichotomy,
every subsequence of $(f_n)$ has a further pointwise convergent
subsequence. By the Main Theorem in \cite{Ro2}, the closure of
$\{f_n\}$ in $\rr^X$ is in $\bb_1(X)$, and so, $\bff\in\SRC(X)$.
Conversely assume that $T_\bff$ is ill-founded. There exists
$L=\{l_0<l_1<..\}\in[\nn]$ such that the sequence $(f_{l_n})$ is equivalent
to the standard basis of $\ell_1$. By the fact that $(f_n)$
is uniformly bounded and Lebesgue's dominated convergence theorem,
we get that the sequence $(f_{l_n})$ has no pointwise
convergent subsequence. This implies that the closure of
$\{f_n\}$ in $\rr^X$ contains a homeomorphic copy of $\beta\nn$,
and so, $\bff\notin\SRC(X)$. It follows that the map $\Phi$
determines a Borel reduction of $\SRC(X)$ to $\wf$. Hence the set
$\SRC(X)$ is $\PB^1_1$ and that the map $\phi_X:\SRC(X)\to\omega_1$
defined by $\phi_X(\bff)=o(T_\bff)$ is a $\PB^1_1$-rank on $\SRC(X)$.

We proceed to show that the set $\SRC$ is $\PB^1_1$-true.
Denote by $\phi$ the canonical $\PB^1_1$-rank $\phi_{2^\nn}$
on $\SRC$ defined above. In order to prove that $\SRC$ is
$\PB^1_1$-true, by \cite[Theorem 35.23]{Kechris}, it is
enough to show that $\sup\{\phi(\bff): \bff\in\SRC\}=\omega_1$.
In the argument below we shall use the following simple fact.
\begin{fact}
\label{srcf1} Let $X, Y$ be compact metrizable spaces and
$e:X\to Y$ a continuous onto map. Let $\bs\in\SRC(Y)$ and
define $\mathbf{g}=(g_n)\in C(X)^\nn$ by $g_n(x)=f_n(e(x))$
for every $x\in X$ and every $n\in\nn$. Then $\mathbf{g}\in\SRC(X)$
and $\phi_Y(\bff)=\phi_X(\mathbf{g})$.
\end{fact}
Now let $\mathcal{F}$ be a family of finite subsets of $\nn$
which is hereditary (i.e. if $F\in\mathcal{F}$ and $G\subseteq F$,
then $G\in\mathcal{F}$) and compact in the pointwise topology
(i.e. compact in $2^\nn$). To every such family $\ff$ one
associates its order $o(\ff)$, which is simply the order of the
downwards closed, well-founded tree $T_\ff$ on $\nn$ defined by
\[ s\in T_\ff \Leftrightarrow s=(n_0<...<n_k)\in \FIN(\nn)
\text{ and } \{n_0,...,n_{k}\}\in\ff. \]
Such families are well-studied in Combinatorics and Functional
Analysis and a detailed exposition can be found in \cite{AT}.
What we need is the simple fact that for every countable ordinal
$\xi$ one can find a compact hereditary family $\ff$ with
$o(\ff)\geq\xi$.

So, fix a countable ordinal $\xi$ and let $\ff$ be a compact
hereditary family with $o(\ff)\geq\xi$. We will additionally
assume that $\{n\}\in\ff$ for all $n\in\nn$. Define $\pi_n^\ff:
\ff\to\rr$ by $\pi_n^\ff(F)=\chi_F(n)$ for all $F\in\ff$.
Clearly for every $n\in\nn$ we have $\pi_n^\ff\in C(\ff)$ and
$\|\pi^\ff_n\|_{\infty}=1$. Moreover, as the family $\ff$ contains
all singletons, we get $\pi^\ff_n\neq \pi_m^\ff$ if $n\neq m$.
It is easy to see that the sequence $(\pi_n^\ff)$ converges
pointwise to 0, and so, $(\pi_n^\ff)\in \SRC(\ff)$.
\begin{claim}
\label{c-order} We have that $\phi_\ff\big((\pi_n^\ff)\big)\geq o(\ff)\geq\xi$.
\end{claim}
\noindent \textit{Proof of Claim \ref{c-order}.} The proof
is essentially based on the fact that $\ff$ is hereditary.
Indeed, notice that if $F=\{n_0<...<n_k\}\in\ff$, then
$(e_i)_{i=0}^k\stackrel{2}{\sim} (\pi^\ff_{n_i})_{i=0}^k$
or equivalently $F\in T^2_{(\pi_n^\ff)}$. To see this, fix
$F=\{n_0<...<n_k\}\in\ff$ and let $a_0,...,a_k\in\rr$
arbitrary. We set
\[ I_{+}=\big\{ i\in\{0,...,k\}:a_i\geq 0\big\} \text{ and }
I_{-}=\{0,...,k\}\setminus I_{+}.\]
Then, either $\sum_{i\in I_+} a_i\geq \frac{1}{2} \sum_{i=0}^k
|a_i|$ or $-\sum_{i\in I_-} a_i \geq \frac{1}{2} \sum_{i=0}^k |a_i|$.
Assume that the second case occurs (the argument is symmetric).
Let $F_-=\{n_i:i\in I_-\}\subseteq F\in\ff$.
Then $F_-\in\ff$ as $\ff$ is hereditary.
Now observe that
\[ \frac{1}{2} \sum_{i=0}^k |a_i| \leq -\sum_{i\in I_-} a_i =
\big| \sum_{i=0}^k a_i\pi_{n_i}^\ff(F_-) \big| \leq \big\|
\sum_{i=0}^k a_i\pi_{n_i}^\ff \big\|_{\infty}\leq 2\sum_{i=0}^k
|a_i|. \]
It follows by the above discussion that the identity map
$\mathrm{Id}:T_\ff\to T^2_{(\pi_n^\ff)}$ is a well-defined
monotone map. The claim is proved. \hfill $\lozenge$
\medskip

By Fact \ref{srcf1} and Claim \ref{c-order}, we conclude that
$\sup\{ \phi(\bff):\bff\in\SRC\}=\omega_1$ and the entire proof
is completed.
\end{proof}


\section{Topological and strong embedding}

Consider the classes $\SRC(X)$ and $\SRC(Y)$, where $X$ and $Y$ are
compact metrizable spaces, as they were coded in the previous section.
There is a canonical notion of embedding between elements of $\SRC(X)$
and $\SRC(Y)$ defined as follows.
\begin{defn}
\label{srcd1} Let $X,Y$ be compact metrizable spaces, $\bs\in\SRC(X)$
and $\bsg\in\SRC(Y)$. We say that $\bgg$ topologically embeds into
$\bff$, in symbols $\bgg<\bff$, if there exists a homeomorphic
embedding of $\overline{\{g_n\}}^p$ into $\overline{\{f_n\}}^p$.
\end{defn}
Clearly the notion of topological embedding is natural and
meaningful, as $\bff_1<\bff_2$ and $\bff_2<\bff_3$ imply that
$\bff_1<\bff_3$. However, in this setting, one also has a canonical
$\PB^1_1$-rank on $\SRC$ and any notion of embedding between elements
of $\SRC$ should be coherent with this rank, in the sense that if
$\bgg<\bff$, then $\phi_Y(\bgg)\leq\phi_X(\bff)$. Unfortunately,
the topological embedding is not strong enough in order to have
this property.
\begin{examp}
\label{srcex1} Let $\ff_1$ and $\ff_2$ be two compact hereditary
families of finite subsets of $\nn$. As in the proof of Lemma
\ref{srcl1}, consider the sequences $(\pi^{\ff_1}_n)\in \SRC(\ff_1)$
and $(\pi^{\ff_2}_n)\in\SRC(\ff_2)$. Both of them are pointwise
convergent to 0. Hence, they are topologically equivalent and clearly
bi-embedable. However, it is easy to see that the corresponding
ranks of the two sequences depend only on the order of the families
$\ff_1$ and $\ff_2$, and so, they are totally unrelated.
\end{examp}
We are going to strengthen the notion of topological embedding
between the elements of $\SRC$. To motivate our definition,
let $\bs, \bsg\in\SRC$ and assume that both $(f_n)$ and $(g_n)$
are Schauder basic sequences. In this case, the most
natural notion of embedding is that of equivalence, i.e. $\bgg$
embeds into $\bff$ if there exists $L=\{l_0<l_1<...\}\in [\nn]$
such that the sequence $(g_n)$ is equivalent to $(f_{l_n})$.
It is easy to verify that, in this case, we do have that
$\phi(\bgg)\leq\phi(\bff)$. Although not every $\bff\in\SRC$ is
a Schauder basic sequence, there is a metric relation we can
impose on $\bff$ and $\bgg$ which incorporates the above
observation.
\begin{defn}
\label{srcd2} Let $X,Y$ be compact metrizable spaces, $\bs\in\SRC(X)$
and $\bsg\in\SRC(Y)$. We say that $\bgg$ strongly embeds into $\bff$,
in symbols $\bgg\prec\bff$, if $\bgg$ topologically embeds into $\bff$
and, moreover, if for every $\ee>0$ there exists $L_\ee=\{l_0<l_1<...\}
\in[\nn]$ such that for every $k\in\nn$ and every $a_0,..., a_k\in\rr$
we have
\begin{equation}
\label{srce1} \Big| \max_{0\leq i\leq k} \big\| \sum_{n=0}^i a_ng_n
\big\|_{\infty} - \big\| \sum_{n=0}^k a_n f_{l_n}\big\|_{\infty} \Big|
\leq \ee \sum_{n=0}^k \frac{|a_n|}{2^{n+1}}.
\end{equation}
\end{defn}
Below we gather the basic properties of the notion of
strong embedding.
\begin{prop}
\label{srcp1} Let $X$ and $Y$ be compact metrizable spaces. The following hold.
\begin{enumerate}
\item[(i)] If $\bff\in\SRC(X)$ and $\bgg\in\SRC(Y)$ with $\bgg\prec\bff$,
then $\bgg<\bff$.
\item[(ii)] If $\bff\in\SRC(X)$, $\bgg\in\SRC(Y)$ with $\bgg\prec\bff$
and the sequence $(g_n)$ is a normalized Schauder basic sequence,
then there exists $L=\{l_0<l_1<...\}\in[\nn]$ such that the sequence
$(f_{l_n})$ is Schauder basic and equivalent to $(g_n)$.
\item[(iii)] If $\bff_1\prec\bff_2$ and $\bff_2\prec\bff_3$, then
$\bff_1\prec\bff_3$.
\item[(iv)] If $\bff\in\SRC(X)$ and $\bgg\in\SRC(Y)$ with $\bgg\prec\bff$,
then $\phi_Y(\bgg)\leq\phi_X(\bff)$.
\item[(v)] Let $Z$ be a compact metrizable space and $e:Z\to X$ onto
continuous. Let $\bs\in\SRC(X)$ and define, as in Fact \ref{srcf1},
$\mathbf{h}=(h_n)\in\SRC(Z)$ by $h_n(z)=f_n(e(z))$ for every $n\in\nn$
and every $z\in Z$. If $\bgg\in\SRC(Y)$ is such that $\bgg\prec\bff$,
then $\bgg\prec\mathbf{h}$.
\end{enumerate}
\end{prop}
\begin{proof}
(i) It is straightforward.\\
(ii) Let $K\geq 1$ be the basis constant of $(g_n)$. We are going to
show that there exists $L=\{l_0<l_1<...\}\in[\nn]$ such that $(g_n)$ is
$2K$-equivalent to $(f_{l_n})$. Indeed, let $0<\ee<\frac{1}{4K}$
and select $L_\ee=\{l_0<l_1<...\}\in[\nn]$ such that inequality
(\ref{srce1}) is satisfied. Let $k\in\nn$ and $a_0,...,a_k\in\rr$.
Notice that
\begin{equation}
\label{srce2} \big\| \sum_{n=0}^k a_n g_n\big\|_\infty \leq \max_{0\leq i\leq k}
\big\| \sum_{n=0}^i a_n g_n\big\|_\infty \leq K \big\| \sum_{n=0}^k
a_n g_n\big\|_\infty.
\end{equation}
Moreover, for every $m\in\{0,...,k\}$ we have
\begin{equation}
\label{srce3} |a_m|\leq 2K \big\| \sum_{n=0}^k a_n g_n\big\|_\infty
\end{equation}
as $(g_n)$ is a normalized Schauder basic sequence (see \cite{LT}).
Plugging in inequalities (\ref{srce2}) and (\ref{srce3}) into (\ref{srce1})
we get
\begin{eqnarray*}
\big\| \sum_{n=0}^k a_n f_{l_n}\big\|_{\infty} & \leq & K \big\| \sum_{n=0}^k
a_n g_n\big\|_\infty + 2K\ee \big\| \sum_{n=0}^k a_n g_n\big\|_\infty \\
& \leq & 2K \big\| \sum_{n=0}^k a_n g_n\big\|_\infty
\end{eqnarray*}
by the choice of $\ee$. Arguing similarly, we see that
\[ \frac{1}{2K} \big\| \sum_{n=0}^k a_n g_n\big\|_\infty \leq
\big\| \sum_{n=0}^k a_n f_{l_n}\big\|_{\infty}. \]
Thus $(g_n)$ is $2K$-equivalent to $(f_{l_n})$, as desired. \\
(iii) It is a simple calculation, similar to that of part (ii), and
we prefer not to bother the reader with it. \\
(iv) Let $d\geq 1$. We fix $\ee>0$ with $\ee<\frac{1}{2d}$
and we select $L_\ee=\{l_0<l_1<...\}\in [\nn]$ such that inequality (\ref{srce1})
is satisfied. For every $s=(m_0<...<m_k)\in T^d_\bgg$ we set
$t_s=(l_{m_0}<...<l_{m_k})\in \FIN(\nn)$. Observe that for
every $k\in\nn$ and every $a_0,...,a_k\in\rr$ we have
\begin{eqnarray*}
2d\sum_{n=0}^k |a_n|\geq \big\| \sum_{n=0}^k a_n f_{l_{m_n}} \big\|_\infty
& \geq & \max_{0\leq i\leq k} \big\| \sum_{n=0}^i a_n g_{m_n} \big\|_\infty -
\ee \sum_{n=0}^k |a_n|\\
& \geq & \big\| \sum_{n=0}^k a_n g_{m_n} \big\|_\infty - \ee \sum_{n=0}^k |a_n|  \\
& \geq & \frac{1}{d} \sum_{n=0}^k |a_n| - \frac{1}{2d} \sum_{n=0}^k |a_n| \\
& = & \frac{1}{2d}\sum_{n=0}^k |a_n|.
\end{eqnarray*}
This yields that $t_s\in T^{2d}_\bff$. It follows that the map
$s\mapsto t_s$ is a monotone map from $T^d_{\bgg}$ to $T^{2d}_{\bff}$.
Hence $o(T^d_\bgg)\leq o(T^{2d}_\bff)$. As $d$ was arbitrary, this
implies that $\phi_Y(\bgg)\leq \phi_X(\bff)$, as desired. \\
(v) It is also straightforward, as the map $e$ induces an isometric
embedding of $C(X)$ into $C(Z)$.
\end{proof}
We are going to present another property of the notion of
strong embedding which has a Banach space theoretic flavor.
To this end, we give the following definition.
\begin{defn}
\label{d1} Let $E$ be a compact metrizable space and $\bgg=(g_n)$ be
a bounded sequence in $C(E)$. By $X_\bgg$ we shall denote the
completion of $c_{00}(\nn)$ under the norm
\begin{equation}
\| x\|_\bgg =\sup\Big\{ \big\| \sum_{n=0}^k x(n)g_n\big\|_\infty:
k\in\nn  \Big\}.
\end{equation}
\end{defn}
We shall denote by $(e_n^\bgg)$ the standard Hamel basis of
$c_{00}(\nn)$ regarded as a sequence in $X_\bgg$. Let us
isolate some elementary properties of $(e_n^\bgg)$.
\begin{enumerate}
\item[(P1)] The sequence $(e_n^\bgg)$ is a monotone basis
of $X_\bgg$. Moreover, $(e_n^\bgg)$ is normalized (respectively
seminormalized) if and only if $(g_n)$ is.
\item[(P2)] If $(g_n)$ is Schauder basic with
basis constant $K$, then $(e_n^\bgg)$ is $K$-equivalent
to $(g_n)$.
\end{enumerate}
Less trivial is the fact (which we will see in the next section)
that $\bgg\in\SRC(E)$ if and only if $(e^\bgg_n)$ is in $\SRC(K)$,
where $K$ is the closed unit ball of $X^*_{\bgg}$ with the
weak* topology. In light of property (P2) above, the
sequence $(e^\bgg_n)$ can be regarded as a sort of
``approximation" of $(g_n)$ by a Schauder basic sequence.

The following proposition relates the strong embedding
of a sequence $\bgg=(g_n)$ into a sequence $\bff=(f_n)_n$
with the existence of subsequences of $(f_n)$ which
are ``almost isometric" to $(e^\bgg_n)$. Its proof,
which is left to the interested reader, is based
on similar arguments as the proof of Proposition
\ref{srcp1}.
\begin{prop}
\label{propnew} Let $X$ and $Y$ be compact metrizable spaces,
$\bgg=(g_n)\in \SRC(X)$ and $\bff=(f_n)\in\SRC(Y)$.
If $\bgg$ strongly embeds into $\bff$, then for every
$\ee>0$ there exists $L_\ee=\{l_0<l_1<...\}\in[\nn]$
such that $(e^\bgg_n)$ is $(1+\ee)$-equivalent to
$(f_{l_n})$.
\end{prop}


\section{The main result}

We are ready to state and prove the strong boundedness result for
the class $\SRC$.
\begin{thm}
\label{srct1} Let $A$ be an analytic subset of $\SRC$. Then there
exists $\bff\in\SRC$ such that for every $\bgg\in A$ we have
$\bgg\prec \bff$.
\end{thm}
We record the following consequence of Theorem \ref{srct1} and
Proposition \ref{propnew}.
\begin{cor}
\label{cornew} Let $X$ be a compact metrizable space and $\bgg=(g_n)\in\SRC(X)$.
Then $(e^\bgg_n)$ is in $\SRC(K)$, where $K$ is the closed unit ball
of $X^*_{\bgg}$ with the weak* topology.
\end{cor}
We proceed to the proof of Theorem \ref{srct1}.
\begin{proof}[Proof of Theorem \ref{srct1}]
We fix a norm dense sequence $(d_n)$ in the closed unit
ball of $C(2^\nn)$ such that $d_n\neq d_m$ if $n\neq m$ and
$d_n\neq 0$ for every $n\in\nn$. We also fix a sequence $(D_n)$ of
infinite subsets of $\nn$ such that $D_n\cap D_m=\varnothing$
if $n\neq m$ and $\nn=\bigcup_n D_n$. Let $A$ be an analytic
subset of $\SRC$ and define $\tilde{A}\subseteq
\nn^\nn$ by
\begin{eqnarray*}
\sg\in\tilde{A} & \Leftrightarrow & \exists \bsg\in A \ \exists \ee>0
\text{ such that}\\
& & \big[ \forall n \ \forall k \ \big(k\in D_n \Rightarrow
\| g_n-d_{\sg(k)}\|_\infty \leq \frac{\ee}{2^{k+1}}\big) \big] \text{ and }\\
& & \big[\forall n \ \forall m \ \big( n\neq m\Rightarrow \sg(n)\neq \sg(m)
\big)\big].
\end{eqnarray*}
Then $\tilde{A}$ is $\SB^1_1$. Let $T$ be the unique downwards
closed, pruned tree on $\nn\times\nn$ such that $\tilde{A}=
\mathrm{proj}[T]$. We define a sequence $(h_t)_{t\in T}$
in $C(2^\nn)$ as follows. If $t=(\varnothing,\varnothing)$,
then we set $h_t=0$. If $t\in T$ with $t\neq (\varnothing,\varnothing)$,
then $t=(s,w)$ with $s=(n_0,...,n_m) \in\bt$. We set $h_t=d_{n_m}$.
Clearly $\|h_t\|_\infty\leq 1$ for every $t\in T$. We notice the
following properties of the sequence $(h_t)_{t\in T}$.
\begin{enumerate}
\item[(P1)] For every $\sg\in [T]$ there exists $\bsg\in A$ and
$\ee>0$ such that for every $n\in\nn$ and every $k\geq 1$ with $k-1\in D_n$
we have $\|g_n-h_{\sg|k}\|_\infty\leq\frac{\ee}{2^{k}}$.
\item[(P2)] For every $\bsg\in A$ and every $\ee>0$ there exists
$\sg\in [T]$ such that for every $n\in\nn$ and every $k\geq 1$ with $k-1\in D_n$
we have $\|g_n-h_{\sg|k}\|_\infty\leq\frac{\ee}{2^{k}}$.
\end{enumerate}
We pick an embedding $\phi:T\to\ct$ such that for all $t, t'\in T$
we have $\phi(t)\sqsubset \phi(t')$ if and only if $t\sqsubset t'$.
Let also $e:T\to\nn$ be a bijection such that $e(t)< e(t')$ if
$t\sqsubset t'$ for all $t, t'\in T$. We enumerate the nodes of $T$
as $(t_n)$ according to $e$. Now for every $n\in\nn$ we define
$f_n:2^\nn\times 2^\nn\to\rr$ by
\begin{equation}
f_n(\sg_1,\sg_2)= \chi_{V_{\phi(t_n)}}(\sg_1) \cdot h_{t_n}(\sg_2)
\end{equation}
where $V_{\phi(t_n)}=\{\sg\in 2^\nn: \phi(t_n)\sqsubset \sg\}$. Clearly
$f_n\in C(2^\nn\times2^\nn)$ and $\|f_n\|_\infty\leq 1$ for all $n\in\nn$.
Moreover, it is easy to check that $f_n\neq f_m$ if $n\neq m$.

It will be convenient to adopt the following notation. For every
function $g:2^\nn\to\rr$ and every $\tau\in 2^\nn$ by
$g*\tau:2^\nn\times 2^\nn\to\rr$ we shall denote the function defined by
$g*\tau(\sg_1,\sg_2)=\delta_\tau(\sg_1)\cdot g(\sg_2)$ for every
$(\sg_1,\sg_2)\in 2^\nn\times 2^\nn$ ($\delta_\tau$ stands for
the Dirac function at $\tau$).
\begin{claim}
\label{cl1} We have $(f_n)\in\SRC(2^\nn\times 2^\nn)$.
\end{claim}
\noindent \textit{Proof of Claim \ref{cl1}.} By the Main Theorem
in \cite{Ro2}, it is enough to show that every subsequence
of $(f_n)$ has a further pointwise convergent subsequence.
So, let $N\in[\nn]$ arbitrary. By Ramsey's theorem, there exists
$M\in[N]$ such that the family $\{\phi(t_n): n\in M\}$ either
consists of pairwise incomparable nodes, or of pairwise comparable.
In the first case we see that the sequence $(f_n)_{n\in M}$ is
pointwise convergent to 0. In the second case we notice that,
by the properties of $\phi$ and the enumeration of $T$,
for every $n,m\in M$ with $n<m$ we have $t_n\sqsubset t_m$. It follows
that there exists $\sg\in [T]$ such that $t_n\sqsubset\sg$ for every
$n\in M$. We may also assume that $t_n\neq (\varnothing, \varnothing)$
for all $n\in M$. By property (P1) above, there exist $\bsg\in A$, $\ee>0$
and a sequence $(k_n)_{n\in M}$ in $\nn$ (with possible repetitions)
such that $\|g_{k_n}-h_{t_n}\|_\infty\leq \frac{\ee}{2^{|t_n|}}$ for all
$n\in M$. As $\bgg\in \SRC$, there exists $L\in [M]$ such that the sequence
$(g_{k_n})_{n\in L}$ is pointwise convergent to a Baire-1 function $g$.
By the fact that $\lim_{n\in L} |t_n|=\infty$, we get that the sequence
$(h_{t_n})_{n\in L}$ is also pointwise convergent to $g$. Finally
notice that the sequence $(\chi_{V_{\phi(t_n)}})_{n\in L}$
converges pointwise to $\delta_{\tau}$, where $\tau$ is the unique
element of $2^\nn$ determined by the infinite chain $\{\phi(t_n):n\in L\}$
of $\ct$. It follows that the sequence $(f_n)_{n\in L}$ is pointwise
convergent to the function $g*\tau$. The claim is proved. \hfill $\lozenge$
\begin{claim}
\label{cl2} For every $\bsg\in A$, $\bgg$ topologically embeds into $(f_n)$.
\end{claim}
\noindent \textit{Proof of Claim \ref{cl2}.} Let $\bsg\in A$.
By property (P2) above, there exists $\sg\in [T]$
such that for every $n\in\nn$ and every $k\geq 1$ with $k-1\in D_n$
we have $\|g_n-h_{\sg|k}\|_\infty\leq\frac{1}{2^{k}}$.
By the choice of $\phi$, we see that there exists a unique
$\tau\in 2^\nn$ such that $\phi(\sg|k) \sqsubset \tau$ for all
$k\in\nn$. Fix $n_0\in\nn$. By the fact that there exist infinitely
many $k$ with $\|g_{n_0}-h_{\sg|k}\|_\infty\leq\frac{1}{2^{k}}$,
arguing as in Claim \ref{cl1} we get that the function
$g_{n_0}*\tau$ belongs to the closure of $\{f_n\}$
in $\rr^{2^\nn \times 2^\nn}$. It follows that the map
\[ \overline{\{g_n\}}^p \ni g \mapsto g*\tau\in
\overline{\{f_n\}}^p\]
is a homeomorphic embedding and the claim is proved.
\hfill $\lozenge$
\begin{claim}
\label{cl3} For every $\bsg\in A$, $\bgg$ strongly embeds into
$(f_n)$.
\end{claim}
\noindent \textit{Proof of Claim \ref{cl3}.} Fix $\bsg\in A$.
By Claim \ref{cl2}, it is enough to show that for every
$\ee>0$ there exists $L_\ee=\{l_0<l_1<...\}\in [\nn]$
such that inequality (\ref{srce1}) is satisfied for
$(g_n)$ and $(f_{l_n})$. So, let $\ee>0$ arbitrary. Invoking
property (P2) above, we see that there exist $\sg\in [T]$
such that for every $n\in\nn$ and every $k\geq 1$ with $k-1\in D_n$
we have $\|g_n-h_{\sg|k}\|_\infty\leq \frac{\ee}{2^{k}}$.
There exists $D=\{m_0<m_1<...\}\in [\nn]$ with $m_0\geq 1$ and
such that $m_n-1\in D_n$ for every $n\in\nn$. By the properties
of the enumeration $e$ of $T$, there exists $L=\{l_0<l_1<...\}\in
[\nn]$ such that $t_{l_n}=\sg|m_n$ for every $n\in\nn$.
We isolate, for future use, the following facts.
\begin{enumerate}
\item[(F1)] For every $n\in\nn$ we have $\|g_n-h_{t_{l_n}}
\|_\infty\leq \frac{\ee}{2^{m_n}}\leq \frac{\ee}{2^{n+1}}$.
\item[(F2)] For every $n,m\in \nn$ with $n<m$ we have
$t_{l_n}\sqsubset t_{l_m}$.
\end{enumerate}
We claim that the sequences $(g_n)$ and $(f_{l_n})$
satisfy inequality (\ref{srce1}) for the given $\ee>0$.
Indeed, let $k\in \nn$ and $a_0,...,a_k\in\rr$. By (F1)
above, for every $i\in\{0,...,k\}$ we have
\[ \Big| \ \big\| \sum_{n=0}^i a_n g_n\big\|_{\infty} - \big\|
\sum_{n=0}^i a_n h_{t_{l_n}} \big\|_{\infty} \Big| \leq \ee
\sum_{n=0}^i \frac{|a_n|}{2^{n+1}}. \]
This implies that
\[ \Big| \max_{0\leq i\leq k} \big\| \sum_{n=0}^i a_n g_n
\big\|_{\infty} - \max_{0\leq i\leq k} \big\| \sum_{n=0}^i
a_n h_{t_{l_n}} \big\|_{\infty} \Big| \leq \ee \sum_{n=0}^k
\frac{|a_n|}{2^{n+1}}. \]
The above inequality is a consequence of the following
elementary fact. If $(r_i)_{i=0}^k$, $(\theta_i)_{i=0}^k$
and $(\delta_i)_{i=0}^k$ are finite sequences of positive
reals such that $|r_i-\theta_i|\leq \delta_i$ for all
$i\in\{0,...,k\}$, then
\[ \big| \max_{0\leq i\leq k} r_i - \max_{0\leq i\leq k}
\theta_i \big| \leq \max_{0\leq i \leq k} \delta_i. \]
So the claim will be proved once we show that
\[ \max_{0\leq i\leq k} \big\| \sum_{n=0}^i a_n h_{t_{l_n}}
\big\|_{\infty} = \big\| \sum_{n=0}^k a_n f_{l_n} \big\|_\infty. \]
To this end we argue as follows.
For every $t\in T$ the function $h_t$ is continuous.
So there exist $j\in\{0,...,k\}$ and $\sg_2\in 2^\nn$ such that
\[ \max_{0\leq i\leq k} \big\| \sum_{n=0}^i a_n h_{t_{l_n}}
\big\|_{\infty}= \big| \sum_{n=0}^j a_n h_{t_{l_n}}(\sg_2) \big|. \]
By (F2), we have $t_{l_0}\sqsubset ...\sqsubset t_{l_k}$.
Hence, by the properties of $\phi$, we see that $\phi(t_{l_0})
\sqsubset ...\sqsubset \phi(t_{l_k})$. It follows that there
exists $\sg_1\in 2^\nn$ such that $\chi_{V_{\phi(t_{l_n})}}(\sg_1)=1$
if $n\in \{0,...,j\}$ while $\chi_{V_{\phi(t_{l_n})}}(\sg_1)=0$
otherwise. So
\[ \big\| \sum_{n=0}^k a_n f_{l_n} \big\|_\infty \geq
\big| \sum_{n=0}^k a_n f_{l_n}(\sg_1,\sg_2) \big| =
\big| \sum_{n=0}^j a_n h_{t_{l_n}}(\sg_2) \big|. \]
Conversely, let $(\sg_3,\sg_4)\in 2^\nn\times 2^\nn$ be such that
\[ \big\| \sum_{n=0}^k a_n f_{l_n} \big\|_\infty =
\big| \sum_{n=0}^k a_n f_{l_n}(\sg_3,\sg_4) \big|. \]
We notice that if $\chi_{V_{\phi(t_{l_n})}}(\sg_3)=1$ for
some $n\in\nn$, then for every $m\in\nn$ with $m\leq n$
we also have that $\chi_{V_{\phi(t_{l_m})}}(\sg_3)=1$.
Hence, there exists $p\in\{0,...,k\}$ such that
$\chi_{V_{\phi(t_{l_n})}}(\sg_3)=1$ if $n\in \{0,...,p\}$
while $\chi_{V_{\phi(t_{l_n})}}(\sg_3)=0$ otherwise.
This implies that
\begin{eqnarray*}
\big\| \sum_{n=0}^k a_n f_{l_n} \big\|_\infty  & = &
\big| \sum_{n=0}^k a_n f_{l_n}(\sg_3,\sg_4) \big|
= \big| \sum_{n=0}^p a_n f_{l_n}(\sg_3,\sg_4)\big| \\
& = & \big| \sum_{n=0}^p a_n h_{t_{l_n}}(\sg_4)\big|
\leq \max_{0\leq i\leq k} \big\| \sum_{n=0}^i a_n
h_{t_{l_n}} \big\|_\infty
\end{eqnarray*}
and the claim is proved.  \hfill $\lozenge$
\medskip

\noindent As $2^\nn\times 2^\nn$ is homeomorphic to $2^\nn$, by
Claims \ref{cl1} and \ref{cl3} and invoking Proposition \ref{srcp1}(v),
the proof of the theorem is completed.
\end{proof}


\end{document}